\documentclass[11pt]{amsart}
\usepackage{xcolor}
\usepackage[a4paper,asymmetric]{geometry}
\usepackage{latexsym}
\usepackage{epsfig}
\usepackage{amsthm,euscript}
\usepackage{amssymb}
\usepackage{amsfonts}
\usepackage{amsmath}
\usepackage{epstopdf}
\newtheorem{theorem}{Theorem}[section]
\newtheorem{thm}[theorem]{Theorem}

\newtheorem{lem}[theorem]{Lemma}
\newtheorem{proposition}[theorem]{Proposition}

\newtheorem{corollary}[theorem]{Corollary}

\theoremstyle{definition}

\newtheorem{defn}[theorem]{Definition}

\theoremstyle{remark}
\newtheorem{remark}[theorem]{Remark}
\newtheorem{rem}[theorem]{Remark}
\numberwithin{equation}{section}

 \DeclareMathAlphabet{\mathpzc}{OT1}{pzc}{m}{it}

\def\a{\alpha}

\def\ep{\varepsilon}
\def\P{\mathbb P}
\def\dim{{\rm dim}_{_{\rm H}}}
\def\dimh{{\rm dim}_{_{\rm H}}}
\def\dimp{{\rm dim}_{_{\rm P}}}
\def\l{{\langle}}
\def\r{\rangle}
\def\I{{\rm 1 \hskip-2.9truept l}}

 \newcommand{\E}{\mathbb{E}}            

 \newcommand{\R}{\mathbb{R}}
 
 \newcommand{\PP}{\mathbb{P}}
 
 \newcommand{\mcl}{\mathcal}
 
 \newcommand{\Be}{\begin{equation}}
 \newcommand{\Ee}{\end{equation}}
 \newcommand{\Bs}{\begin{split}}
 \newcommand{\Es}{\end{split}}
  \newcommand{\Bes}{\begin{equation*}}
 \newcommand{\Ees}{\end{equation*}}
 \newcommand{\BT}{\begin{thm}}
 \newcommand{\ET}{\end{thm}}
 \newcommand{\Bp}{\begin{proof}}
 \newcommand{\Ep}{\end{proof}}
 \newcommand{\BL}{\begin{lem}}
 \newcommand{\EL}{\end{lem}}
 \newcommand{\BP}{\begin{proposition}}
 \newcommand{\EP}{\end{proposition}}
 \newcommand{\BC}{\begin{corollary}}
 \newcommand{\EC}{\end{corollary}}
 \newcommand{\BR}{\begin{rem}}
 \newcommand{\ER}{\end{rem}}
 \newcommand{\BD}{\begin{defn}}
 \newcommand{\ED}{\end{defn}}
 \newcommand{\BI}{\begin{itemize}}
 \newcommand{\EI}{\end{itemize}}

 \newcommand{\dif}{{\rm d}}

\begin{document}
\title
[Uniform dimension results for a family of Markov processes]
{Uniform dimension results for a family of Markov processes}

\author[X. Sun ]{Xiaobin {\sc Sun}}
\address{School of Mathematics and Statistics\\Jiangsu Normal University\\
Xuzhou 221116, China}
\email{xbsun@jsnu.edu.cn}

\author[Y. Xiao]{Yimin Xiao}
\address{Department of Statistics and Probability,
Michigan State University, East Lansing, MI 48824}
\email{xiao@stt.msu.edu}

\author[L. Xu]{Lihu Xu }
\address{Department of Mathematics,
Faculty of Science and Technology,
University of Macau, E11
Avenida da Universidade, Taipa,
Macau, China;
UM Zhuhai Research Institute, Zhuhai, China.}
\email{lihuxu@umac.mo}

\author[J. Zhai]{Jianliang Zhai}
\address{School of Mathematical Science,
University of Science and Technology of China, Hefei, 230026, China}
\email{zhaijl@ustc.edu.cn}

\maketitle

\begin{abstract} \label{abstract}
In this paper we prove uniform Hausdorff and packing
dimension results for the images of a large family of Markov processes.
The main tools are the two covering principles in \cite{X2004}.
As applications, uniform Hausdorff and packing dimension results
for self-similar Markov processes, certain classes of L\'evy processes, stable jump diffusions and
non-symmetric stable-type processes are obtained.
\end{abstract}

\section{Preliminaries}

Fractal properties of Brownian motion and more general L\'evy processes have
been studied extensively. We refer to the recent books of M\"orters and Peres
\cite{MP}, Schilling and Partzsch \cite{Schilling-Partzsch2014}, B\"ottcher, Schilling and Wang
\cite{BSW13}, the survey papers \cite{Taylor86, X2004}, and more recent articles
\cite{KX05, KX15, KSW15} for further information.

Let $X= \{X(t), t \in \R_+\}$ be a stable L\'evy process in $\R^d$  of index $\a \ (0 < \a \le 2)$.
For any Borel set $E \subseteq \R_+=[0, \infty)$, Blumenthal and Getoor \cite{BG60} obtained
the Hausdorff dimension of the image set $X(E)$, namely,
\begin{equation}\label{Eq:BG60}
\dim X(E) = \min \big\{d,\   \a\, \dim E\big\} \ \ \text{  a.s.},
\end{equation}
where $\dim $ denotes Hausdorff dimension; see Falconer \cite{Fal},
or \cite{MP,Taylor86,X2004} for the definitions and properties of
Hausdorff measure and Hausdorff dimension.

Result (\ref{Eq:BG60}) has been extended and strengthened in various
directions by many authors; see \cite{Taylor86, X2004} and the references
therein for a historical account and information on development on (mostly)
L\'evy processes. In particular, Hawkes and Pruitt \cite[Theorem 4.1]{HP}
established a uniform version of (\ref{Eq:BG60}): If $X$
is a strictly stable L\'evy process of index $\a$ in $\R^d$ and $d \ge \a$,
then there exists a single null probability event outside of which (\ref{Eq:BG60})
holds simultaneously for all Borel sets $E \subseteq \R_+$. The first such
uniform dimension result was due to Kaufman \cite{K68} for planar Brownian
motion (i.e., $\alpha = 2$ and $d = 2$). A short and easily accessible redaction of
Kaufman's original argument can be found in \cite{Schilling-Partzsch2014}.
Such uniform dimension results are useful in many situations
because it allows $E$ to be a random set (cf. \cite{BCR,MP} for some applications).
Perkins and Taylor \cite{PerkinsTaylor} further proved uniform Hausdorff and
packing  measure results for strictly stable L\'evy processes. As a consequence,
they proved a packing dimension analogue of \cite[Theorem 4.1]{HP}: If $X$
is a strictly stable L\'evy process of index $\a$ in $\R^d$ and $d \ge \a$,
then with probability 1,
\begin{equation}\label{Eq:PT87}
\dimp X(E) =\a\, \dimp E \ \ \hbox{ for all Borel sets }\, E \subseteq \R_+,
\end{equation}
where $\dimp $ denotes packing dimension (cf. e.g., \cite{Fal,MP,Taylor86,X2004}).
Note that, a strictly $\a$-stable L\'evy process is self-similar with index
$H= 1/\alpha$, and is an important representative among self-similar processes
and random fractals.

In recent years, there has been increasing interest in constructing and studying more general Markov processes
related to L\'evy processes. A large class of Markov processes are generated by pseudo-differential operators
\cite{BSW13, Jacob2001-6, Jacob-Schilling2001, Schilling-Schnurr2010}, their
corresponding transition probabilities and heat kernel estimates have been studied in, for example,
\cite{Kolok00b, Kuhn2017a, Kuhn2017b,Negoro94}. Also, we refer to \cite{B88} for the martingale problem
of pure jump Markov processes, \cite{C2009, CK, CZ} for stable-like processes 
related to Dirichlet forms and their heat kernel estimates.
Many natural questions regarding sample path and fractal properties arise for such Markov processes.
Xiao \cite{X2004} gives a comprehensive survey on fractal properties of L\'evy or more general Markov
processes before 2004. A lot of progress has been made since then.
See, for example,  \cite{KSX12, KX03,KX05,KX08,KX09,KX15} for various results on L\'evy processes,
\cite{C2009, CK, Yang}  for Hausdorff dimension of the range or graph of stable-like processes,
\cite{KSW15,Knopova-Sch} for results on Hausdorff dimensions of
the image, level and collision sets  of a class of Feller process generated by a pseudo-differential operator.
However,  many interesting problems described \cite{X2004} are still open.



The main purpose of this paper is to provide a general method for establishing uniform
Hausdorff and packing dimension results for more general Markov processes.
In particular, we extend the methods of Hawkes \cite{HJ} and Hawkes and Pruitt
\cite{HP} for stable L\'evy processes (see also Pruitt \cite{Pruitt1975}) to a large
family of Markov processes including more general L\'evy processes, the stable-like
processes or stable jump diffusions in \cite{CK, CZ, Kolok00b,Negoro94}. The key
technical tools are the covering principles
in Section 2, which improve Lemmas 8.1 and 8.2 in Xiao \cite{X2004}. We apply
them to show the main result of this paper, Theorem \ref{main theorem}, in Section 3.
In Section 4, we apply Theorem \ref{main theorem} to L\'evy processes and stable-type
processes. We mention that Benjamini, Chen and Rohde
\cite{BCR} studied the normally reflected Brownian motion (RBM) in a class of
non-smooth domains and proved their uniform Hausdorff dimension result by using the
uniform H\"older continuity of RBM,  the aforementioned Kaufman's theorem for Brownian
motion and a subordination argument. Moreover, they indicated in Remark 3.10 in \cite{BCR}
that similar result still holds for the stable-like processes in \cite{CK} by using the
covering principle in Xiao \cite{X2004}. The scope of the present paper is
a lot broader, and we expect that the main result of this paper will also be useful
for studying fractal sets related to intersections and multiple points of Markov processes.



In the rest of this paper, we assume that $X = \{X(t), \, t \in \R_+ , \PP^x\}$
is a time-homogeneous Markov process with state space $\R^d$, defined on some
probability space $(\Omega, \mathcal {F}, \PP)$, and satisfies the strong Markov property.
We assume that $X$ is separable and its sample paths are almost surely right
continuous and have left limit at every $t \in \R_+$ (such a sample function
will be called cadlag).

An unspecified positive and finite constant will be denoted by $C$, which may
be different in each appearance. More specific constants are numbered by $K_1,
K_2, \ldots$ and $C_1, C_2, \ldots.$

\section{The covering principles}

In the review article \cite[Lemmas 8.1 and 8.2]{X2004}, Xiao stated without
proof two covering principles that extend respectively Lemma 3.1
in  \cite{HP}  and Lemma 3 in \cite{HJ}  for L\'evy processes (see also Lemmas
1 and 2 in  \cite{Pruitt1975}), and suggested that they are useful for proving
uniform Hausdorff and packing dimension results for the images of a general
Markov process. In this section, we weaken the conditions of Lemmas 8.1 and
8.2 in \cite{X2004} and provide proofs.


The following lemma is useful for proving upper bounds for the Hausdorff and
packing dimensions of the image of a Markov
process. Its proof is a modification of that  of Lemma 3.1 in \cite{HP}, and we
provide it for the sake of completeness. Moreover, in Proposition \ref{Prop:C1}
below, we will provide a convenient way to verify condition \eqref{max}.

\begin{lem}\label{lem 2.2} Let $X=\{X(t),t\in\mathbb{R}_{+}, \mathbb{P}^{x}\}$
be a time homogenous strong Markov process in $\mathbb{R}^{d}$. Let $\{t_n, n\geq 1\}$
be a sequence of positive real numbers such that $\sum^{\infty}_{n=1}t_n^p<\infty$
for some $p>0$, and let $\mathcal{C}_n$ be a class of $N_n$ intervals in
$\mathbb{R}_{+}$ of length $t_n$ with $\log N_n= O(1)|\log t_n|$. If there is a
sequence $\{\theta_n\}$ of positive numbers such that for all $x \in \R^d$,
\begin{eqnarray}\label{max}
\mathbb{P}^{x}\bigg\{\sup_{0\leq s\leq t_n}|X(s)-x|\geq \theta_n \bigg\}
\leq K_1 t_n^{\delta},
\end{eqnarray}
where $K_1$ and $\delta$ are some positive constants, then there exists a positive
integer $K_2$, depending on $p$ and $\delta$ only, such that,
with $\mathbb{P}^{x}$-probability one, for $n$ large enough, $X(I)$
can be covered by $K_2$ balls of radius $\theta_n$ whenever $I\in \mathcal{C}_n$.
\end{lem}

\begin{proof} Let $I\in \mathcal{C}_n$ and write it as $I=[a, a+t_n]$.
Let $\tau_0=a$ and, for all $j \ge 1$, define
$$
\tau_j=\inf\{s> \tau_{j-1}:\, |X_s-X_{\tau_{j-1}}|>\theta_n\},
$$
with the convention $\inf \emptyset = \infty$. It is easy to see
\
\Be  \label{e:CovTau}
\{X(I, \omega)\ \text{cannot be covered by} \ k\ \text{balls of radius}\
\theta_n \}\subseteq \{\tau_k-\tau_0\leq t_n\}.
\Ee
Moreover, by the strong Markov property and \eqref{max},
\begin{eqnarray*}
\mathbb{P}^{x}\{\tau_k-\tau_0\leq t_n\} &\leq& \mathbb{E}^x\left\{\mathbb{E}^{x}
\left[1_{\{\tau_k-\tau_{k-1}\leq t_n\}}1_{\{\tau_{k-1}-\tau_0\leq t_n\}}
|\mathcal{F}_{\tau_{k-1}}\right]\right\}  \\
&=&\mathbb{E}^x\left\{\mathbb{E}^{x}\left[1_{\{\tau_k-\tau_{k-1}\leq t_n\}}
|\mathcal{F}_{\tau_{k-1}}\right]1_{\{\tau_{k-1}-\tau_0\leq t_n\}}\right\}  \\
&\leq &\sup_{y\in\mathbb{R}^d}\mathbb{P}^{y}\left\{\sup_{0\leq s\leq t_n}|X(s)-y|\geq
\theta_n\right\}\mathbb{E}^x\left\{1_{\{\tau_{k-1}-\tau_0\leq t_n\}}\right\}\\
&\leq &K_1 t_n^{\delta} \mathbb{P}^{x}\{\tau_{k-1}-\tau_0\leq t_n\}.
\end{eqnarray*}
Using the above argument recursively, we obtain that for all $n \ge 1$,
\
\Be  \label{e:CovTau1}
\mathbb{P}^{x}\{\tau_k-\tau_0\leq t_n\} \le K_1^k t_n^{k\delta}.
\Ee
Define events
$$
A^k_n:=\big\{\exists \ I\in \mathcal{C}_n \text{ such that} \ X(I,\omega) \
\text{cannot be covered by} \ k \ \text{balls of radius} \  \theta_n\big\}.
$$
Since $\log N_n= O(1)|\log t_n|$, i.e., there exist positive constant $C$,
such that $N_n\leq C t_n^{-C}$, as $k$ is large enough (say,
$\delta k - C \ge p$), we obtain from \eqref{e:CovTau1} that
\begin{eqnarray*}
\sum^{\infty}_{n=1}\mathbb{P}^{x}(A^k_n)
& \leq & \sum^{\infty}_{n=1}N_n K_1^k t_n^{k\delta} \\
& \leq & C K_1^k\sum^{\infty}_{n=1}t_n^{-C+\delta k}
\leq C K_1^k\sum^{\infty}_{n=1}t_n^{p}<\infty.
\end{eqnarray*}
Hence, the Borel-Cantelli lemma yields the desired result.
The proof is complete.
\end{proof}

For obtaining lower bounds for the Hausdorff and packing dimensions of
the image of a Markov process, one can apply the following lemma.
Observe that the condition
\eqref{min} is significantly weaker than (8.7) in Lemma 8.2
in \cite{X2004} (which is usually satisfied only if $X$ is transient)
and is easier to verify (see \eqref{delay hitting} below).

\begin{lem} \label{lower bounds} Let $X=\{X(t), t\in\mathbb{R}_{+},
\mathbb{P}^{x}\}$ be a time homogenous strong Markov process in
$\mathbb{R}^d$. Let $\{r_n, n\geq 1\}$ be a sequence of positive
numbers with $\sum^{\infty}_{n=1}r_n^p<\infty$ for some $p>0$, and
let $\mathcal{D}_n$ be a class of $N_n$ balls (or cubes) of diameter
$r_n$ in $\mathbb{R}^d$ with $\log N_n= O(1)|\log r_n|$. If, for
every constant $T > 0$, there exists a sequence $\{t_n\}$ of positive
numbers and constants $K_3$ and $\delta>0$ such that
\begin{eqnarray} \label{min}
\mathbb{P}^{x} \bigg\{\inf_{t_n\leq s\le T}|X(s)-x|\leq r_n\bigg\}
\leq K_3 r_n^{\delta},
\quad \forall \, x \in \mathbb{R}^d,
\end{eqnarray}
then there exists a constant $K_{4}$, depending on $p$ and $\delta$ only,
such that, with $\mathbb{P}^{x}$-probability one, for $n$ large enough,
$X^{-1}(B) \cap [0, T]$ can be covered by at most $K_{4}$ intervals of
length $t_n$, whenever $B\in \mathcal{D}_n$.
\end{lem}
\begin{proof}
Let $B\in\mathcal{D}_n$ and assume $B=B(z, \frac{r_n}{2})$ for some
$z\in\mathbb{R}^d$ since the diameter of $B$ is $r_n$. Let $\tau_0=0$
and, for any $k \ge 1$, define
$$
\tau_k=\inf\Big\{t\geq \tau_{k-1}+t_n, |X_t-z|\leq \frac{r_n}{2}\Big\},
$$
with the convention $\inf \emptyset=\infty$. It is easy to see
$$
\big\{t: X_t\in B\big\}\subseteq \bigcup^{\infty}_{i=0}\,[\tau_i, \tau_i+t_n),
$$
which implies
$$
\big\{\tau_k> T \big\}\subseteq \big\{X^{-1}(B, \omega) \cap [0, T]\ \text{can
be covered by} \  k\  \text{intervals of length}\ t_n \big\}.
$$
Hence
$$
\big\{X^{-1}(B, \omega)\cap [0, T]\  \text{cannot be covered by} \ k\
\text{intervals of length}\, t_n \big\}\subseteq \big\{\tau_k\le T\big\}.
$$
By the strong Markov property, \eqref{min} and the fact that $X(\tau_{k-1})
\in B$ as $\tau_{k-1} \le T$, we obtain
\begin{eqnarray*}
\mathbb{P}^{x}\{\tau_k\le T \}
&\leq& \mathbb{P}^x\{\tau_k\le T|\tau_{k-1} \le T\}
\mathbb{P}^{x}\{\tau_{k-1} \le T\}\\
&\leq&\sup_{y\in B}\mathbb{P}^{y}\left\{\inf_{t_n\leq s\le T}|X(s)-z|\leq \frac{r_n}{2}\right\}\mathbb{P}^{x}\{\tau_{k-1} \le T\}\\
&=&\sup_{y\in B}\mathbb{P}^{y}\left\{\inf_{t_n\leq s<T}|X(s)-y+y-z|\leq \frac{r_n}{2}\right\}\mathbb{P}^{x}\{\tau_{k-1}\le T\}\\
&\leq&\sup_{y\in B}\mathbb{P}^{y}\left\{\inf_{t_n\leq s\le T}|X(s)-y|\leq r_n\right\}\mathbb{P}^{x}\{\tau_{k-1} \le T\}\\
&\leq& K_3 r_n^{\delta}\mathbb{P}^{x}\{\tau_{k-1} \le T\}.
\end{eqnarray*}
By iterating the above argument,  we obtain
\Be
\mathbb{P}^{x}\{\tau_k \le T\} \le K_3^k r_n^{k\delta}.\nonumber
\Ee
Define the events
\begin{equation*}
\begin{split}
A^k_n:=\big\{\omega \in \Omega: &\exists B\in \mathcal{D}_n \ \ \text{ s.t.}\ \
X^{-1}(B,\omega)\cap [0, T] \ \ \text{cannot be covered by} \\
& \ k \  \text{intervals of length}\ \ t_n\big\}.
\end{split}
\end{equation*}
Since $\log N_n= O(1)|\log r_n|$, i.e., there exists a positive constant $C$
such that $N_n\leq r_n^{-C}$ for all integers $n$, we see that  for $k$ large enough
(say, $\delta k - C \ge p$),
\begin{eqnarray*}
\sum^{\infty}_{n=1}\mathbb{P}^{x}(A^k_n)\leq \sum^{\infty}_{n=1}N_n K_3^k r_n^{k\delta}
\leq K_3^k\sum^{\infty}_{n=1}r_n^{-C+\delta k}
\leq K_3^k\sum^{\infty}_{n=1}r_n^{p}<\infty.
\end{eqnarray*}
Hence, the conclusion of Lemma \ref{lower bounds}  follows from the Borel-Cantelli lemma.
\end{proof}


\section{Main result}


The objective of this section is to establish uniform Hausdorff and packing dimension
results for the images of a time homogeneous Markov process $X = \{X(t), t\geq 0,
\mathbb{P}^x\}$ with values in $\R^d$. For any Borel set $A$ in $\mathbb{R}^d$,
denote by $P(t,x,A):=\mathbb{P}^{x}(X_t\in A)$ the transition probability of $X$. We
state the following assumptions,
where (A1) will be used for deriving uniform upper bounds, and (A2) or (A3)
for uniform lower bounds.

\vskip 3mm

(A1)\, There is a constant $H > 0$ such that for any $\gamma \in (0, H)$,
there exist constants $C>0$, $\eta>0$ and $t_0\in(0,1)$ such that for all
$x\in\mathbb{R}^d$ and $0<t\leq t_0$,
\begin{equation}\label{Eq:supX}
\mathbb{P}^{x}\bigg\{\sup_{0\leq s\leq t}|X(s)-x|\geq t^{\gamma}\bigg\}\leq C t^{\eta}.
\end{equation}

\vskip 3mm
(A2)  There is a sequence of vectors of non-negative numbers $J=\{(\ep_n, \zeta_n), n \ge 1\}$
such that $\ep_n \to 0$ and $\zeta_n \to 0$ as $n \to \infty$, and has the following property
(for simplicity of notation, we omit the subscript $n$): For any  $(\varepsilon, \zeta) \in J$ and
constant  $T > 0$,  there exist positive constants $C_1$, $C_2$, and $r_0 \leq 1$  such that
 for all  $0 < r \le r_0$,  $x, y \in \R^d$ with $ |y-x|\leq r$, and all $ 0 < t  \le T$,
\begin{equation}\label{Eq:P1}
P(t, y, B(x,r)) \ge C_1 \min\bigg\{1,\,  \Big(\frac{r} {t^{H-\zeta}}\Big)^{d+\ep}\bigg\};
\end{equation}
and
\begin{equation}\label{Eq:P2}
P(t, x, B(x,r)) \le   C_2 \min\bigg\{1,\, \Big(\frac{r} {t^{H+\zeta}}\Big)^{d-\ep}\bigg\}.
\end{equation}

\vskip 3mm
(A3)  We strengthen (A2) by further assuming that \eqref{Eq:P2} holds for all $ t  > 0$.

\vskip 2mm




Condition (A2) is quite general due to the flexibility in choosing arbitrarily small
constants  $\ep$ and $\zeta$, in order for (\ref{Eq:P1}) and (\ref{Eq:P2}) to hold.
This condition can be satisfied by a large class of
Markov  processes such as those with a bounded transition density and an
approximate scaling property; see Section \ref{sec:examples} for some interesting
examples. (A3) is slightly stronger than (A2), which is needed for our subordination
argument in proving Theorem \ref{main theorem} below in the critical case of $1 = Hd$.

Condition (A1) is less obvious. In the following, we give a
sufficient condition for it to hold. For any $h \ge 0$ and $a>0$,
similar to Manstavi\v cius \cite{Man04}, we consider the function
\begin{equation}\label{Eq:alphaf}
\a(h, a)  = \sup\Big\{ P\big(s, x, B(x, a)^c\big): \, x \in \R^d,\,
0\le s \le h\Big\},
\end{equation}
where $B(x, a)^c  = \{y\in \R^d: |y-x| \ge a\}$. The function $h
\mapsto \a(h, a)$ carries a lot of information about regularity properties
of the sample paths of the Markov process $X$. For example,
Kinney \cite{Kinney53} showed that if for each fixed $a>0$, $\a (h, a)
\to 0$ as $h \to 0$ then the sample function $X(t)$ is almost surely
cadlag; and if $\a (h, a) = o(h)$ as $h \to 0$ for every fixed $a >
0$ then the sample function $X(t)$ is almost surely continuous. See
Manstavi\v cius \cite{Man04} and the references therein for further
information.

For a given  constant $H > 0$, a Markov process $X$ is said to belong
to the class $\widetilde{\EuScript M} (H)$ if there exist positive and finite
constants $C$, $\beta$, $h_0$ and $a_0$, depending
on $d$ and $H$ only, such that the following property holds: For all $h \in (0,
h_0)$ and  $ a \in (0,a_0)$ such that $h a^{-1/H} < 1$,  we have
\begin{equation}\label{Def:Mclass}
\a(h, a) \le C\, \bigg(\frac{h} {a^{1/H}}\bigg)^{\beta}.
\end{equation}
Condition (\ref{Def:Mclass}) is the same as  (1.1)
in Manstavi\v cius \cite{Man04} for the class ${\EuScript M} (\beta,\gamma)$
with $\gamma =  \beta/H$.
We mention that \cite[Corollary 5.10]{Schilling-Schnurr2010} proved that
the solution of certain SDE driven by a L\'evy process
belongs to the class ${\EuScript M} (1,\gamma)$ of Manstavi\v cius
\cite{Man04} for suitable choice of $\gamma$, thus derived a result on the
$\gamma$-variation of the solution.
\vskip 2.5mm

The following sufficient condition for (A1)  is often convenient to use
(cf. Theorem 4.1 below).
\begin{proposition}\label{Prop:C1}
Let $X=\{X(t), t \in \R_+, \PP^x\}$ be a separable, time-homogeneous
Markov process taking values in $\R^d$. If $X$ belongs to the class
$\widetilde{\EuScript M} (H)$, then for any $\ep \in (0, 1)$ and
$\gamma \in (0, H)$,  $X$ satisfies  (\ref{Eq:supX}) with $\eta = \beta
(1 - \frac \gamma H).$
\end{proposition}

\begin{proof}\ We make use of the following Ottaviani-type inequality
(cf. Gikhman and Skorohod \cite[page 420]{GS74},
or Manstavi\v cius \cite{Man04}): For all $x \in \R^d$, all $h >0$ and
$a >0$ such that $\a(h, a/2) <1$,
\begin{equation}\label{Eq:Ott}
\PP^x\bigg\{ \sup_{0 \le s\le h} |X(s) - X(0)| > a \bigg\} \le
\frac{\PP^x\big\{ |X(h) - X(0)| > a/2 \big\}} {1 - \a(h, a/2)}.
\end{equation}
For any $\gamma \in (0, H)$, it follows
from (\ref{Eq:Ott}) with $a = h^\gamma$  and (\ref{Def:Mclass}) that for
$h$ small enough,
\begin{equation}\label{Eq:Ott2}
\begin{split}
\PP^x \bigg\{ \sup_{0 \le s\le h} |X(s) - X(0)| > h^\gamma \bigg\} &\le
C \PP^x\Big\{ |X(h) - X(0)| > h^\gamma/2 \Big\} \le Ch^{\beta(1 - \frac \gamma H)}.
\end{split}
\end{equation}
This proves the proposition.
\end{proof}
\vskip 0.2cm

Most examples given in Section 4 are L\'evy or L\'evy-type processes.
For these processes, the maximal tail probability $\mathbb{P}^x
\big\{\sup_{s\in[0,t]}|X_s-x|\geq r\big\}$ in (\ref{Eq:Ott2}) has
been studied by several authors.  Pruitt \cite{Pru81} established
an upper bound for the maximal probability for a general L\'evy
process in terms of its L\'evy measure. Schilling \cite{Schilling 1998}
and B\"ottcher et al. \cite{BSW13}  extended Pruitt's result to
L\'evy-type processes and proved an upper bound in terms of the
symbol of the process. The following proposition is taken from
\cite[Corollary 5.2]{BSW13},
which can be applied to verify (A1) for L\'evy-type processes.
We remark that K\"uhn \cite[Lemma 3.2]{Kuhn2017a} has proved recently that
the inequality (\ref{eq Maximal estimate}) still holds if $t$ is
a stopping time, with the $t$ on the right-hand side replaced by $\E(t)$.
We thank the referee for pointing out these results to us.

\begin{proposition}\label{Prop:C1-1}
Let $X = \{X(t), t\geq 0\}$ be a L\'evy-type process with a symbol $q(x,\xi):\mathbb{R}^d\times
\mathbb{R}^d\rightarrow \mathbb{C}$ given by
$$
q(x,\xi)= -ib(x)\cdot \xi +\frac{1}{2}\xi\cdot Q(x)\xi +\int_{\mathbb{R}^d \setminus  \{0\}}
\big(1-e^{iy\cdot\xi}+iy\cdot\xi 1_{(0,1]}(|y|) \big)\nu(x,dy),
$$
where for each fixed $x\in\mathbb{R}^d$, $(b(x),Q(x),\nu(x,dy))$ is a L\'evy triplet, i.e.
$b(x)\in\mathbb{R}^d$, $Q(x)\in\mathbb{R}^{d\times d}$ is a symmetric positive semidefinite matrix
and $\nu(x,dy)$ is a measure on $(\mathbb{R}^d \setminus \{0\},\mathcal{B}(\mathbb{R}^d \setminus  \{0\}))$
such that $\int_{\mathbb{R}^d \setminus  \{0\}}(|y|^2\wedge 1)\nu(x,dy)<\infty$.
Then,  there exists a constant $C>0$ such that
\begin{eqnarray}\label{eq Maximal estimate}
\mathbb{P}^x\bigg\{\sup_{s\in[0,t]}|X_s-x|\geq r\bigg\}\leq Ct\sup_{|y-x|\leq r}\sup_{|\xi|\leq 1/r}|q(y,\xi)|.
\end{eqnarray}
\end{proposition}

\vskip 0.2cm

Our main theorem of this paper is the following uniform Hausdorff and
packing dimension result for the images of $X$.

\begin{thm} \label{main theorem}
Let $X=\{X(t), t \in \R_+, \PP^x\}$ be a time homogeneous Markov
process in $\mathbb{R}^d$ and satisfies Conditions (A1). Assume either
(i) $1 < H d$ and (A2) hold; or (ii) $1 = H d$ and (A3) hold.
Then for all $x \in \R^d$,
\begin{equation}\label{Eq:HdimU}
\mathbb{P}^x\Big\{\dimh X(E)= \frac{1}{H}\dimh E \ \text{for all Borel sets}
\  E \subseteq [0, \infty)\Big\}=1
\end{equation}
and
\begin{equation}\label{Eq:PdimU}
\mathbb{P}^x\Big\{\dimp X(E)= \frac{1}{H}\dimp E \ \text{for all Borel sets}
\  E \subseteq [0, \infty)\Big\}=1.
\end{equation}
\end{thm}
\begin{proof}
We will only prove the Hausdorff dimension result (\ref{Eq:HdimU}).
The proof of \eqref{Eq:PdimU}, which is based on the connection between packing
dimension and upper box-counting dimension (cf. \cite{Fal}), is similar and
hence omitted.

The proof of  is divided into two parts. Namely, we prove the upper
and lower bounds for $\dimh X(E)$, respectively.

\vskip 2mm
{\it Part 1 (Uniform upper bound)}.  By the $\sigma$-stability of
Hausdorff dimension (cf. \cite{Fal}), it suffices to consider Borel sets
$E\subseteq [0, L]$ for all fixed integers $L$. For simplicity, we
take $L = 1$ in this proof. Let $ \gamma \in (0, H)$ be a constant,
$t_n=2^{-n}$ and
$$
\mathcal{C}_n= \big\{ [(j-1)t_n, j t_n]: j=1,2,\ldots, 2^n \big\}.
$$
By Condition (A1)  we get that for all $x \in \R^d$,
\begin{eqnarray*}
\mathbb{P}^{x}\left\{\sup_{0\leq s\leq t_n}|X(s)-x|\geq  t_n^{\gamma}\right\}
\leq C t_n^{\eta}.
\end{eqnarray*}
Thus, by Lemma \ref{lem 2.2}, with probability one under $\mathbb{P}^x$, as
$n$ is sufficiently large, $X(I)$ can be covered by $K_2$ balls of radius
$\theta_n:= t_n^{\gamma}$ for all intervals $I\in \mathcal{C}_n$.
\vskip 3mm

Let $\chi=\dimh E$. Then, for any $\delta>0$, there exists a sequence of intervals
$\{F_i\subseteq [0,1],\, i\in\mathbb{N}\}$ such that $d_i:={\rm diam}(F_i)\le \ep$,
\Be  \label{e:HauDef}
E \subseteq \bigcup_{i}F_i \ \ \text{ and }\ \   \sum_i d_i^{\chi+\delta}\le 1.
\Ee
We choose $n_i$ so that $\frac{t_{n_i}}2\leq d_i\leq t_{n_i},$
thus $F_i$ is  contained in at most two intervals of $\mathcal{C}_{n_{i}}$.
Consequently, $X(F_i)$ can be covered by $2K_2$ balls of radius $\theta_{n_i}$,
which are denoted by $B_{i,1}, ..., B_{i,2K_2}$.
Hence,
$$X(E)\subseteq\bigcup_{i}\bigcup^{2K_2}_{j=1}B_{i,j}.$$
Further observe that, by using \eqref{e:HauDef}, we have
\Bes
\begin{split}
\sum_{i} \sum_{j=1}^{2K_2} \left[{\rm diam}(B_{i,j})\right]^{(\chi+\delta)/\gamma}
& \le 2K_2\, \sum_{i} t_{n_i}^{\chi+\delta}  \\
& \le 2K_2\, \sum_{i} (2d_i)^{\chi+\delta}  \le 2^{1 +\chi+\delta} K_2.
\end{split}
\Ees
This yields $\dim X(E)\leq (\chi+\delta)/\gamma$. Letting $\delta\downarrow 0$
and $ \gamma \uparrow H$ yields
$$\dim X(E)\leq \frac{1}{H}\dim E.$$
\vskip 2mm

{\it Part 2 (Uniform lower bound)}. In order to prove
\begin{equation}\label{Eq:HdimL}
\mathbb{P}^x\Big\{\dimh X(E)\ge \frac{1}{H}\dimh E \ \text{for all Borel sets}
\  E \subseteq [0, \infty)\Big\}=1,
\end{equation}
we will treat the two cases (i) $1 < H d$ and  (ii) $1 = H d$ separately.

We observe that the  inequality in (\ref{Eq:HdimL}) follows from the
following claim: For all $x \in \R^d$,
\begin{equation}\label{Eq:LB1}
\PP^x \big\{\dim X^{-1}(F) \le {H}\dim F \ \ \hbox{ for all }\, F\subseteq \R^d\big\} = 1,
\end{equation}
by taking $F = X(E)$. Moreover, by the $\sigma$-stability of Hausdorff dimension,
(\ref{Eq:LB1}) is equivalent to: For all constants $T >0$ and all integers $m = 1, 2, ...$,
\begin{equation}\label{Eq:LB2}
\PP^x \Big\{\dim \big(X^{-1}(F)\cap [0, \,T]\big) \le {H}\dim F \ \ \hbox{ for all }\,
F\subseteq [-m, m]^d \Big\} = 1.
\end{equation}
Hence, in order to prove (\ref{Eq:HdimL}) for the case (i), it suffices to prove
\eqref{Eq:LB2} for  all fixed constant $T>0$ and positive integer $m$. This will be
done by using the covering principle in Lemma \ref{lower bounds}. Recall that in case
(i) we assume (A2) holds. Specifically,  \eqref{Eq:P1} and \eqref{Eq:P2}
hold for all $0 < t \le 2T$.

By applying \cite[Proposition 2.1]{X98}, which is an extension of Theorem 1.1 in \cite{K97},
we get
\begin{equation}  \label{delay hitting}
\mathbb{P}^{x}\Big\{\inf_{t \le s \le T} |X(s)-x|\leq r\Big\}
\leq \frac{\int^{2T}_{t}P(s, x, B(x,r))ds}{\inf_{|y-x|\leq r}
\int^{T-t}_{0}P(s, y, B(x,r))ds}.
\end{equation}
In order to estimate the denominator of (\ref{delay hitting}), we assume without
loss of generality that $0 < t \le T/2$.
For any $y\in\mathbb{R}^d$ with $|y-x|\leq r\leq r_0$, we use (\ref{Eq:P1})
in (A2) to get
\begin{eqnarray*}
\int^{T-t}_{0}P(s,y, B(x,r))ds \geq C_1 \int^{T/2}_{0} \min\bigg\{1,\,
\Big(\frac{r} {s^{H-\zeta}}\Big)^{d+\ep}\bigg\}\,ds
\ge C r^{\frac 1 {H-\zeta}}.
\end{eqnarray*}
On the other hand, since $1 < Hd$, we choose $(\ep, \zeta) \in J$ such that
$\ep$ and $\zeta$ are small enough so that $1 < (H+ \zeta)(d-\ep)$. By using
(\ref{Eq:P2}) in (A2) we derive that for any $t > r^{1/(H+\zeta)}$,
\[
\begin{split}
\int^{2T}_{t}P(s, x, B(x,r))ds &\leq  C_2
\int^{2T}_{t} \min\bigg\{1,\, \Big(\frac{r}
{s^{H+\zeta}}\Big)^{d-\ep}\bigg\}\, ds\\
&= C_2 \int^{2T}_{t}  \Big(\frac{r}
{s^{H+\zeta}}\Big)^{d-\ep}\, ds \leq C\, r^{d-\ep} t^{1 - (H+\zeta)(d-\ep)}.
\end{split}
\]
We combine the above with \eqref{delay hitting} to see that
for all $x\in\mathbb{R}^d$, $r > 0$ and $t > r^{1/(H+\zeta)}$,
\begin{equation}  \label{delay hitting probability}
\mathbb{P}^{x}\left\{\inf_{t \le s \le T} |X(s)-x|\leq r\right\}
\leq C\, r^{d- \frac 1 {H-\zeta}-\ep} t^{1 - (H+\zeta) (d-\ep)}.
\end{equation}

Now we proceed to prove \eqref{Eq:LB2} for the case $1 < Hd$.
Fix any integer $m \ge 1$ and let $\mathcal{D}^m_n$
be the collection of dyadic cubes in $[-m, m]^d$ of the form
$\prod_{i=1}^d [j_i2^{-n},\, (j_i+1){2^{-n}}]$,
where each $j_i$ is an integer and $j_i\in [-m2^n, m2^n-1]$. It is
easy to see that $\sharp(\mathcal{D}^m_n)=(m2^{n+1})^d.$

Fix a constant $\gamma\in(0, 1/H)$ and then choose $\varepsilon>0$
and $\zeta>0$ sufficiently small  such that  $(\ep, \zeta) \in J$ and
$$
\delta:= \big(d - \frac 1 {H-\zeta} -\ep\big) \big(1 -  \gamma (H+\zeta) \big) -
\frac{2 \gamma \zeta}{H - \zeta} >0.
$$
This is possible since $ (d - \frac 1 H) (1 - \gamma H)> 0$.

Let $r_n=2^{-n}$ and $t_n=2^{-\gamma n}$,
further choose $n$ large enough so that $r_n\leq r_0$. Notice that
$t_n > r_n^{1/(H+\zeta)}$. By \eqref{delay hitting probability} we
verify that for all $x\in\mathbb{R}^d$,
\Be  \label{e:LowBou1}
\mathbb{P}^{x}\left\{\inf_{t_n\leq s\le T}|X_s-x|\leq r_n \right\}\leq
C r_n^{d- \frac 1 {H-\zeta}-\ep} t_n^{1 - (H+\zeta)(d - \ep)} =  C\,r_n^\delta.
\Ee
Hence, Lemma \ref{lower bounds} implies that there
is an integer $K_4$ such that with $\mathbb P^x$-probability one,
for all $n$ large enough (say, $n \ge n_0$) and all dyadic cubes
$B\in\mathcal{D}^m_n$, $X^{-1}(B) \cap [0, T]$ can be covered by at most $K_4$
intervals of length $t_n=2^{-\gamma n}$.
\vskip 2mm

For any Borel set $F \subseteq [-m, m]^d$, let $\theta > \dim F$. Then
there exists a sequence of dyadic cubes $\{B_i, \, i\in\mathbb{N}\}$ in
$[-m, m]^d$ of sides $r_{n_i}= 2^{-n_i}$ such that  $n_i \ge n_0$,
\Be  \label{e:HauDef2}
F \subseteq \bigcup_{i=1}^\infty B_i \ \ \text{ and }\ \
\sum_{i=1}^\infty r_{n_i}^{\theta}\le 1.
\Ee

Since, for every $i$, $X^{-1}(B_i)\cap [0, T]$ can be covered by $K_4$ intervals
$I_{ij}$ of length $t_{n_i}=2^{-\gamma n_i}$, we see that
$$ X^{-1}(F)\cap [0, T] \subseteq \bigcup_{i=1}^\infty \bigcup_{j=1}^{K_4} I_{ij}.$$
Moreover,
\begin{eqnarray*}
\sum_{i=1}^\infty \sum_{j=1}^{K_4} \big[{\rm diam} (I_{ij})\big]^{\theta/\gamma}
\le K_4\sum_{i=1}^\infty r_{n_i}^{\theta}\leq K_4,
\end{eqnarray*}
this implies $\dim \big(X^{-1}(F)\cap [0, T] \big) \leq \frac{\theta}{\gamma}$.
Letting $\gamma\uparrow {1}/{H}$ and $\theta\downarrow \dim F$ yields
\eqref{Eq:LB2} and, thus,  (\ref{Eq:HdimL}) for the case $1 < H d$.

Finally we prove (\ref{Eq:HdimL}) for the case $1 = H d$ by making use of  a
``subordination argument'' that is similar to that in Hawkes \cite{HJ} (see also
Pruitt \cite{Pruitt1975}).

Let $\tau = \{\tau_t, t \ge 0\}$ be a stable subordinator with stability
index $\rho\in(0, 1)$, and independent of the process $X$.  Consider the
Markov process $Y = \{Y_t, t \ge 0\}$ defined by  $Y_t = X(\tau_t)$. It is
easy to see $\{Y_t,\, t\ge0\}$ is a time-homogeneous strong Markov process.
Denote the transition probability of $Y$ by $\tilde{P}(t,x, A) :=\mathbb{P}^x
(Y_t\in A)$. We claim that, if (A3) holds, then $\tilde{P}(t,x, A)$ satisfies
Condition (A2) with $H$ replaced by $H/\rho$. Consequently, because
$1 < Hd/\rho$, we can apply the conclusion of the first part to the process $Y$.

More specifically, we now verify the following claim under assumption (A3):

\vskip 2mm
(A2$'$)\, We can find a sequence $J' = \{(\ep', \zeta')\}$ of arbitrarily
small numbers with the following property: For any  $(\ep',\zeta') \in J'$
and $A > 0$, there exist positive constants $C_3$ and $C_4 $ such that
for all $ 0< r\le r_0$
(as in (A2)), $x, y \in \R^d$ with $ |y-x|\leq r$ and $0 < t  \le A$,
\begin{equation}\label{Eq:P5}
\tilde{P}(t,y,B(x,r)) \ge  C_3 \min\left\{1, \,  \left(\frac{r} {t^{{(H-\zeta')}/{\rho}}}\right)^{d+\ep'}\right\};
\end{equation}
and
\begin{equation}\label{Eq:P6}
\tilde{P}(t,x,B(x,r)) \le  C_4 \min\left\{1,\,  \left(\frac{r} {t^{{(H+\zeta')}/{\rho}}}\right)^{d-\ep'}\right\}.
\end{equation}

To verify (A2$'$),  we denote, for all $t > 0$, the density function of
$\tau_t$ by $p_{\tau_t}$. Then the self-similarity  of $\tau_t$ implies
$p_{\tau_t}(s)=t^{-\frac{1}{\rho}} p_{\tau_1}(t^{-\frac{1}{\rho}}s)$.
Hence
\begin{equation}\label{Eq:tildeP}
\begin{split}
\tilde{P}\big(t,y,B(x,r))&=\PP^y (|Y_t-x|\leq r \big)\\
&=\int^{\infty}_0 \PP^y(|X_{\tau_t}-x|\leq r| \tau_t=s) p_{\tau_t}(s)ds\\
&=\int^{\infty}_0 P(s,y,B(x,r))p_{\tau_t}(s)ds\\
&=\int^{\infty}_0 P(s,y,B(x,r))t^{-\frac{1}{\rho}}p_{\tau_1}(t^{-\frac{1}{\rho}}s)ds\\
&=\int^{\infty}_0 P \left(t^{\frac{1}{\rho}}s,y,B(x,r) \right) p_{\tau_1}(s)ds.
\end{split}
\end{equation}
Consequently, we can make use of Condition (A3) to estimate
$\tilde{P}\big(t,y,B(x,r))$. Recall that $J = \{(\ep, \zeta)\}$ is
the sequence in  (A2).

On one hand, for any constant $A > 0$, we apply (\ref{Eq:P1}) with
$T \ge 2A^{1/\rho}$ to derive that for any $x, y \in \mathbb R^d$
with $|y-x| \le r$, $ 0<r \le r_0$ and $0 < t  \le A$, we have
\begin{equation*}
\begin{split}
\tilde{P}(t,y, B(x,r)) &\geq C_1 \int^{Tt^{-1/\rho}}_0 \min\left\{1, \, \left(\frac{r}
{ (t^{{1}/{\rho}}s)^{H-\zeta}} \right)^{d+\ep}\right\} p_{\tau_1}(s)ds\\
&= C_1 \int^{r^{1/(H-\zeta)} t^{-1/{\rho}}}_0 p_{\tau_1}(s)\,ds \\
&\qquad +C_1\int^{Tt^{-1/\rho}}_{r^{1/(H- \zeta)} t^{-1/{\rho}}} \left(\frac{r}
{t^{{(H-\zeta)}/{\rho}}}\right)^{d+\ep}\frac{p_{\tau_1}(s)}{s^{(H- \zeta)(d+\ep)}}\,ds.
\end{split}
\end{equation*}

It is easy to see that when $r^{1/(H- \zeta)} >t^{1/{\rho}}$,
\begin{equation}\label{3.22}
\tilde{P}(t,y,B(x,r))\geq C_1\int^{1}_0 p_{\tau_1}(s)ds
\end{equation}
and when $r^{1/(H- \zeta)} \le t^{1/{\rho}}$,
\begin{equation}\label{3.23}
\begin{split}
\tilde{P}(t,y,B(x,r))&\geq C_1 \left(\frac{r}{t^{{(H- \zeta)}/{\rho}}}\right)^{d+\ep}
\int^{2}_{1}\frac{1}{s^{(H-\zeta)(d+\ep)}}  p_{\tau_1}(s)ds,
\end{split}
\end{equation}
where we have used that fact that $Tt^{-1/\rho} \ge T A^{-1/\rho} \ge 2$
for all $0 < t \le A$ and the last integral is positive since the density
$p_{\tau_1}(s)$ is positive for $s > 0$.
Then, (\ref{3.22}) and (\ref{3.23}) imply (\ref{Eq:P5}) with $\ep' = \ep$
and $\zeta' = \zeta$.

On the other hand, similarly to \eqref{Eq:tildeP}, we use  (\ref{Eq:P2}),
which is now assumed to hold for all $t > 0$,
to derive that for any  $x\in \mathbb R^d$, $0<r \le r_0$ and $t > 0$,
\begin{equation}
\begin{split}\label{Eq:tildeP2}
\tilde{P}(t,x,B(x,r)) &\leq C_2 \int^{\infty}_0  \min\left\{1,\,  \left(\frac{r}
{(t^{{1}/{\rho}}s)^{H+ \zeta}}\right)^{d-\ep}\right\} p_{\tau_1}(s)ds\\
&= C_2\int^{r^{1/(H+\zeta)} t^{-1/{\rho}}}_0 p_{\tau_1}(s)ds \\
&\qquad +C_2\int^{\infty}_{r^{1/(H+\zeta)} t^{-1/{\rho}}} \left(\frac{r}
{t^{{(H+ \zeta)}/{\rho}}}\right)^{d-\ep}\frac{  p_{\tau_1}(s)}{s^{(H+\zeta)(d-\ep)}}ds.
 \end{split}
\end{equation}
Since $\tilde{P}(t,y,B(x,r))\leq 1$, in order to verify (\ref{Eq:P6}), we
only need to consider the case when $ r^{1/(H+\zeta)} \le t^{1/{\rho}}$.
If $(H+\zeta)(d-\ep) < 1$, then by the boundedness of the density function
$p_{\tau_1}(\cdot)$, i.e., $\sup_{0<s\leq 1}p_{\tau_1}(s)\leq M$ for some $M>0$,
we obtain that
\begin{equation}\label{Eq:324}
\tilde{P}(t,x,B(x,r))\leq C_2 M\frac{r^{1/(H+ \zeta)}}{t^{1/{\rho}}}
+ C_2 \left(\frac{r}{t^{{(H+\zeta)}/{\rho}}}\right)^{d-\ep}\int^{\infty}_{0}
\frac{p_{\tau_1}(s)} {s^{(H+ \zeta)(d-\ep)}}\,ds.
\end{equation}
The last integral is convergent at 0 because $(H+\zeta)(d-\ep) < 1$ and
at infinity because $p_{\tau_1}(s)\sim \frac{C_{\rho}}{s^{1+\rho}}$ as
$s\rightarrow +\infty$.

If $(H+\zeta)(d-\ep) > 1$, then we further split the last integral in
(\ref{Eq:tildeP2}) over $[ r^{1/(H+\zeta)}  t^{-1/{\rho}}, \,1]$ and
$[1, \infty)$, and to derive
\begin{equation}\label{Eq:324b}
\tilde{P}(t,x,B(x,r))\leq C_2' M\frac{r^{1/(H+ \zeta)}}{t^{1/{\rho}}}
+ C_2 \left(\frac{r}{t^{{(H+\zeta)}/{\rho}}}\right)^{d-\ep}\int^{\infty}_{1}
\frac{p_{\tau_1}(s)} {s^{(H+ \zeta)(d-\ep)}}\,ds,
\end{equation}
where $C_2'$ is a finite constant and, again, the last integral is convergent.
The case of $(H+\zeta)(d-\ep) = 1$ can be treated in the same way, and
\eqref{Eq:324b} still holds with an extra factor of
$\log (\frac{r^{1/(H+ \zeta)}}{t^{1/{\rho}}})$  in the first term on the
right hand side, which can be absorbed by choosing $\tilde{\ep}$ below
slightly bigger. However, for simplicity, we ignore this case
because one may choose $J = \{(\ep,\zeta)\}$ so that $(H+\zeta)(d-\ep) \ne 1$
for all $(\ep,\zeta)\in J$.

Notice that, when $1 = Hd$, we can write $\frac 1 {H+\zeta} = d-\tilde{\ep}$,
where $ \tilde{\ep} = \frac{\zeta}{H(H+\zeta)}$.
It follows from (\ref{Eq:324}) and (\ref{Eq:324b}) that for
 $ r^{1/(H+ \zeta)} \le t^{1/{\rho}}$,
\begin{equation}\label{3.25}
\tilde{P}(t,x,B(x,r))\leq C\left(\frac{r} {t^{(H+\zeta)/\rho}} \right)^{d-\ep'}
\end{equation}
for some finite constant $C$, where $\ep' = \max\{\ep, \tilde{\ep}\}$.
Thus we have verified (\ref{Eq:P6}) with $\ep' = \max\{\ep, \tilde{\ep}\}$
and $\zeta' = \zeta$. Moreover, because $0< t \le A$, the inequality in
(\ref{Eq:P5}) remains valid (with a modified constant $C_3$) if we take
$\ep' = \max\{\ep, \tilde{\ep}\}$. Thus, Condition (A2$'$) has been verified.

Because of Condition (A2$'$)  and the fact that $\frac{Hd}{\rho}>1$,  we
can apply the above uniform lower bound result in the case of $Hd>1$
to Markov process $Y$ to obtain
\begin{equation*}
\mathbb{P}\Big\{\dim Y(E)\ge \frac{\rho}{H}\dim E \ \text{for all Borel sets}
\  E \subseteq [0, \infty)\Big\}=1.
\end{equation*}
It follows  that, with probability 1, for all Borel sets
$E \subseteq [0, \infty)$,
\begin{equation}\label{Eq:Ylb}
\dim X(E) \geq \dim X(\tau(B))=\dim Y(B)\geq \frac{\rho}{H} \dim B,
\end{equation}
where  $B=\{t:\tau_t\in E\} = \tau^{-1}(E)$. Even though both $\dim X(E) $ and $\dim B$
in \eqref{Eq:Ylb} are random, they are determined by two independent processes $X$
and $\tau$, respectively. Hence we have
\begin{equation} \label{Eq:Ylb2}
\mathbb{P}\Big\{\dim X(E)\ge \frac{\rho}{H} \| \dim \tau^{-1} (E) \|_\infty
\ \text{for all Borel sets}
\  E \subseteq [0, \infty)\Big\}=1,
\end{equation}
where $\| \cdot\|_\infty$ is the $L^\infty(\mathbb P)$-norm in the underlying
probability space. (A more illustrative way for deriving \eqref{Eq:Ylb2} is
to use the setting of product probability space. Namely,  we assume that $X$
is defined on $\Omega$, $\tau$ is defined on $\Omega'$, then $Y$ is defined on
$\Omega\times \Omega'$ and \eqref{Eq:Ylb} holds for almost all $(\omega, \omega')$.
One can see that \eqref{Eq:Ylb2} follows from \eqref{Eq:Ylb} and Fubini's Theorem.)

Recalling from Hawkes \cite{Hawkes71} that
\begin{equation}\label{Eq:Hawkes71b}
\| \dim \tau^{-1} (E) \|_\infty  = \frac{\rho + \dim E - 1} {\rho},
\end{equation}
 we derive that
 $$
\mathbb{P}\bigg\{ \dim X(E) \geq  \frac{\rho+\dim E-1}{H} \
\text{for all Borel sets}
\  E \subseteq [0, \infty)\bigg\}=1.
$$
Letting $\rho\uparrow 1$ yields (\ref{Eq:HdimL}) for the case $1=H d$.
This completes the proof of Theorem \ref{main theorem}.
\end{proof}



\section{Examples}\label{sec:examples}
Theorem \ref{main theorem} is applicable to a wide class of Markov processes.
In this section, we provide some examples which include self-similar Markov processes,
L\'evy processes, stable jump diffusion processes and non-symmetric stable-type processes.

\subsection{Self-similar Markov processes}

The class of $H$-self-similar ($H$-$s. s.$) Markov processes with values in
$[0, \infty)$ was introduced and studied by Lamperti \cite{L}, who used the
term ``semi-stable'' instead  of ``self-similar''. The $H$-$s. s.$ Markov
processes on $\mathbb R^d$ or $\mathbb R^d \backslash \{0\}$ were investigated
by, in chronicle order, Kiu \cite{Kiu}, Graversen and Vuolle-Apiala
\cite{G-V-A},   Vuolle-Apiala and Graversen \cite{V-A-G}, Vuolle-Apiala \cite{V-A94},
Liu and Xiao \cite{LX98}, Xiao \cite{X98},
Bertoin and Yor \cite{BY05}, Chaumont, et al \cite{CPR}, Alili, et al \cite{ACGZ},
among others.

We recall the definition of $H$-self-similar processes.
Let $(E,\mathcal{B})$ denotes $\mathbb{R}^d$, $\mathbb{R}^d \setminus \{0\}$
or $\mathbb{R}^d_{+}$ with the usual Borel $\sigma$- algebra, $\{e\}$ is a point attached
to $E$ as an isolated point. $\Omega$ denotes
the space of all functions $\omega$ from $[0,\infty)$ to $E\cup\{e\}$ having
the following properties:

\begin{itemize}
\item [(i)]  $\omega(t)=e$ for $ t\geq \tau$, where $\tau=\inf\{t\geq 0: \omega(t)=e\}$;
\item[(ii)]  $\omega$ is right continuous and has a left limit at every $t\in[0,\infty)$.
\end{itemize}

Let $H>0$ be a fixed constant. A time homogeneous Markov process $X=\{X(t),  t \ge 0, \mathbb{P}^x\}$
with state space $E\cup \{e\}$ is called  $H$-self-similar ($H$-s.s.) if its transition probability function
$P(t,x, A)$ satisfies 
\begin{eqnarray}\label{self}
P(t, x, A)=P(rt, r^{H}x, r^{H}A),\quad \text{for all }\  t>0, r>0, x\in E, A\in\mathcal{B}.
\end{eqnarray}
The constant $H$ is called the self-similarity index of $X$.  Condition (\ref{self}) is equivalent to the
statement that for every constant $r>0$ the $\mathbb{P}^x$-distribution of $\{X(t), t \ge 0\}$ is equal to
the $\mathbb{P}^{r^{H}x}$-distribution of $\{r^{-H}X(rt), t \ge 0\}$. 
Important examples of self-similar Markov processes include strictly
$\alpha$-stable L\'evy processes which are $1/\alpha$-s.s., the Bessel processes which form exactly
the class of $1/2$-s.s. diffusions on $(0, \infty)$ (see \cite{RY}).  More examples of $H$-s.s.
Markov processes can be found in \cite{L, G-V-A, Kiu, V-A-G,X98}.

In this section, we take $E = \R^d$ and assume the following two conditions:

\vskip 2mm
(B1)\, There exist positive constants $\beta$, $C$ and
 $a_1$ such that
 \begin{equation} \label{Eq:SS1}
 P(1, x, B(x, a)^c)  \le C a^{-\beta} \quad \hbox{ for all } \, x \in \R^d \hbox{ and }\, a > a_1.
\end{equation}

(B2)\, For any $\varepsilon > 0$ small, there exist positive constants $C_1$ and $C_2$
such that for all $r >0$ and  $x, y \in \mathbb R^d$ with $|x-y| \le r$, we have
\begin{equation} \label{Eq:SS2}
C_1 \min\{1, \,  r^{d+\ep}\} \le  P(1, x, B(y, r)) \le  C_2 \min\{1, \,r^{d-\ep}\}.
\end{equation}

The following theorem provides a uniform version for the Hausdorff dimension
result in \cite{LX98}.
\begin{theorem}\label{Th:SSM}
Let $X = \{X(t), t \in \R_+\}$ be an $H$-s.s. Markov process in $\R^d$
that satisfies conditions (B1) and (B2). If $1 \le H d$, then for all
$x \in \R^d$, with $\P^x$-probability one,
$$
\dim X(E) = \frac 1 {H} \, \dim E \ \hbox{ and }\ \dimp X(E) = \frac 1 H \, \dimp E
$$
for all Borel sets $E \subseteq \R_+$.
\end{theorem}
\begin{proof}  It is sufficient to verify that Conditions (A1) and (A3)
are satisfied. It follows from \eqref{self} and  \eqref{Eq:SS1} that for
any $t > 0$, $x \in \R^d$ and $a > 0$,
\begin{equation}\label{Eq:B1}
P\big(t, x, B(x, a)^c\big) = P\big(1, xt^{-H}, B(xt^{-H}, a t^{-H})^c\big)
\le C (a t^{-H})^{-\beta}
\end{equation}
provided $a t^{-H} \ge a_1$. This implies that $\a(h, a)
\le C(h a^{-1/H})^{H\beta}$ for all $h, a >0$
satisfying $h a^{-1/H} \le a_1^{-1/H}$.
Hence $X$ belongs to the class $\widetilde{\EuScript M} (H)$.
It follows from Proposition \ref{Prop:C1} that Condition (A1) is satisfied.
For verifying (A3), we apply \eqref{self} again, together with the first
inequality in \eqref{Eq:SS2}, to see that for all $t, r > 0$ and
 $x, y \in \R^d$  with $|x-y|\le r$,
\begin{equation}\label{Eq:B2}
P\big(t, x, B(y, r)\big) = P\big(1, xt^{-H}, B(yt^{-H}, r t^{-H})\big)
\ge C_1\min\bigg\{1, \Big(\frac r {t^{H}}\Big)^{d+\ep}\bigg\}.
\end{equation}
Similarly, we have
\begin{equation*}\label{Eq:B3}
P\big(t, x, B(x, r)\big) = P\big(1, xt^{-H}, B(xt^{-H}, r t^{-H})\big)
\le C_2\min\bigg\{1, \Big(\frac r {t^{H}}\Big)^{d-\ep}\bigg\}.
\end{equation*}
Thus, Condition (A3) is satisfied with $J = \{(\ep_n,  0)\}$,
where $\ep_n \downarrow 0$ can be taken arbitrarily. Therefore
the conclusion of Theorem \ref{Th:SSM} follows from Theorem
\ref{main theorem}.
\end{proof}

\subsection{L\'evy processes}

A stochastic process $X =\{X(t),\ t \geq 0\}$ on a probability
space $(\Omega, {\mathcal F}, \PP)$, with values in $\R^d $, is
called a \emph{L\'evy process}, if $X$ has stationary and
independent increments,  $t\mapsto X(t)$ is continuous in
probability and $\PP\{X(0) = 0\} = 1$. It is well known that for
$t\ge 0$, the characteristic function of $X(t)$ is given by
$$
    \E \big[ e^{ i \l \xi, X(t)\r } \big] =
    e^{-t\psi(\xi)},
$$
where, by the L\'evy-Khintchine formula,
\begin{equation}\label{KL}
    \psi(\xi)= i \l \mathsf{a}, \xi\r +\frac{1}{2}
    \l \xi ,\mathsf{\Sigma}
    \xi^{'} \r +\int_{\R^d}\Bigl[1-e^{i \l x, \xi \r}+ \frac{i \l x, \xi
    \r}{1+|x|^2}\Bigr]\ \mathsf{L} (dx),\qquad \forall \xi \in \R^d,
\end{equation}
and $\mathsf{a} \in\R^d$ is fixed, $\mathsf{\Sigma}$ is a
non-negative definite, symmetric, $(d\times d)$ matrix, and
$\mathsf{L}$ is a Borel measure on $\R^d\setminus\{0\}$ that
satisfies
\[
    \int_{\R^d} \frac{|x|^2}{1+|x|^2}\, \mathsf{L}(dx) < \infty .
\]
The function $\psi$ is called the \emph{characteristic or L\'evy exponent}
of $X$, and $\mathsf{L}$ is  the corresponding \emph{L\'evy measure}.
The characteristic exponent $\psi$ plays very important roles in studying
the L\'evy process $X$ and many sample path properties of $X$ can be
described in terms of $\psi$. We also note that
\[
    \Re\, \psi(\xi) \ge 0, \textnormal{ and }\ \
    \Re\,\psi(- \xi) = \Re\, \psi(\xi),
    \qquad \forall\, \xi \in \R^d.
\]
Notice that, if $X$ is symmetric (i.e., $X$ and $-X$ have the same law),
then its L\'evy exponent $\psi(\xi)$ is a nonnegative function.



A L\'evy process $X$ in $\R^d$ is called a \emph{stable} L\'evy
process with index $\alpha \in (0, 2]$ if its L\'evy measure
$\mathsf{L}$  is of the form
\begin{equation}\label{L-measure}
\mathsf{L}(dx) =  \frac{dr} {r^{1 + \a}} \, \nu (dy), \quad
\forall\, x = r y, \ \ (r, y) \in \R_+\times \mathbb{S}_d,
\end{equation}
where $\mathbb{S}_d = \{y \in \R^d:\, |y |=1\}$ is the unit sphere
in $\R^d$ and $\nu(dy)$ is an arbitrary finite Borel measure on
$\mathbb{S}_d$. Stable L\'evy processes in $\R^d$ of index $\a = 1$
are also called \emph{Cauchy processes}. It follows from (\ref{KL})
and (\ref{L-measure}) that the L\'evy exponent $\psi_\a$ of a stable
L\'evy process of index $\a \in (0, 2]$ can be written as
\[
\begin{split}
\psi_\a (\xi) &= \int_{\mathbb{S}_d} |\l \xi, y\r |^\a \Big[ 1 -
i~ {\rm sgn} (\l \xi, y\r) \tan \big( \frac{\pi\a}{2} \big) \Big]
\mathsf{M} (dy) + i \l \xi, A_0\r \quad \hbox{if }\ \a \ne 1,\\
\psi_1 (\xi)  &= \int_{\mathbb{S}_d} |\l \xi, y\r| \Big[ 1 + i~
\frac{\pi} 2 {\rm sgn} (\l \xi, y\r) \log   |\l \xi, y\r|  \Big]
\mathsf{M} (dy) + i \l \xi, A_0 \r ,
\end{split}
\]
where the pair $(\mathsf{M}, A_0)$ is unique, and the measure $\mathsf{M}$,
which depends on $\nu$ in (\ref{L-measure}), is called the \emph{spectral measure}
of $X$. See Samorodnitsky and Taqqu \cite[pp.65--66]{ST94}.

\begin{remark}\label{Re:Lv}
Due to the space-homogeneity of L\'evy processes, Theorem \ref{main theorem} and
its proof can be simplified.
Firstly,  we assume  the following simpler conditions:

\vskip 2mm
(A$2^{''}$)\, For any $\zeta > 0$ and $T > 0$,
there exist positive constants $C_1$, $C_2$, $r_0 \leq 1$
such that for all $ 0< t \le T$ and  $0 < r \le r_0$,
\begin{equation}\label{Eq:LP1a}
{\mathbb P}( |X(t)| \le r ) \ge C_1 \min\bigg\{1,\, \Big(\frac{r}
{t^{H-\zeta}}\Big)^{d}\bigg\};
\end{equation}
and
\begin{equation}\label{Eq:LP1b}
  {\mathbb P}( |X(t)| \le r ) \le
C_2 \min\bigg\{1,\, \Big(\frac{r}
{t^{H+\zeta}}\Big)^{d}\bigg\}.
\end{equation}

\vskip 2mm
(A$3^{''}$)\, We strengthen (A$2^{''}$) by assuming additionally
that \eqref{Eq:LP1b} holds for all $t > 0$.

\vskip 2mm
It is clear that the inequalities in \eqref{Eq:LP1a} and \eqref{Eq:LP1b}
imply Condition (A2) with $J = \{(0, \zeta_n)\}$, where $\zeta_n$ is an
arbitrary sequence with $\zeta_n \downarrow 0$ as $n \to \infty$.
Secondly, instead of \eqref{delay hitting}, we use directly Theorem 1.1
in \cite{K97} to get
\begin{equation} \label{delay hitting2}
\mathbb{P}^{x}\Big\{\inf_{t \le s \le T} |X(s)-x|\leq r\Big\}
\leq \frac{\int^{2T}_{t} \mathbb P(|X(s)| \le 2r )ds}{
\int^{T-t}_{0}\mathbb P(|X(s)| \le r) ds}.
\end{equation}
Note that, due to the space-homogeneity of $X$,  the denominator in
the right hand side of \eqref{delay hitting2} is simpler that that in
\eqref{delay hitting}. One can check that a slightly modified version
of (\ref{delay hitting probability}) holds under condition (A$2^{\prime\prime}$).
Hence, for a space-homogeneous Markov process that satisfies (A1),
the conclusion of Theorem \ref{main theorem} still holds under either
$1 < Hd$ and (A$2^{\prime\prime}$); or $1 = Hd$ and (A$3^{\prime\prime}$).
We will use this modified version of Theorem \ref{main theorem} to L\'evy
processes.
\end{remark}

As we mentioned earlier, for a L\'evy process $X = \{X(t), t \in \R_+\}$,
many of its sample path properties are characterized by the analytic or
asymptotic properties of its characteristic exponent $\psi(\xi)$. In order
to determine the Hausdorff and packing dimension of $X(E)$, we will make
use of the following conditions:

\vskip 2mm
(B3)\, There is a constant $\a \in (0, 2]$ such that the following hold:
\begin{itemize}
\item[(i)] If $0 < \alpha < 2$, then for every $\zeta' \in (0, 2-\alpha)$
we have
\begin{equation}\label{Eq:growth}
K^{-1}_{5} \, |\xi|^{\a - \zeta'} \le \psi(\xi)
\le  K_{5} \, |\xi|^{\alpha + \zeta'},\quad\forall \xi \in \R^d \hbox{
with }\ | \xi|\ge \tau,
\end{equation}
where $K_5 \ge 1 $ and $\tau$ are positive and finite constants.
\item[(ii)]\, If $\alpha = 2$, then for any $\zeta' \in (0, 2)$,
\begin{equation}\label{Eq:growth2}
K^{-1}_{5} \, |\xi|^{2 - \zeta'} \le 
\psi(\xi)
\le  K_{5} \, |\xi|^{2},\quad\forall \xi \in \R^d \hbox{
with }\ | \xi|\ge \tau.
\end{equation}
\end{itemize}

\vskip 2mm
(B4)\, In addition to (B3), we assume that the left inequalities
in \eqref{Eq:growth} and \eqref{Eq:growth2} hold for all $\xi \in \R^d$.

\vskip 2mm
\begin{remark} The following are some remarks about Conditions (B3) and (B4).
\begin{itemize}
\item[(i)]\, Since $\psi(\xi)$ is a negative definite function, the
right inequality in \eqref{Eq:growth2} always holds for $|\xi|\ge 1$. (cf. \cite[p.46]{BF75}).

\item[(ii)]\, Conditions \eqref{Eq:growth} and \eqref{Eq:growth2} are satisfied
by a large class of symmetric L\'evy  processes whose L\'evy measures have certain
(approximate) regularly varying properties at the origin. This can be explicitly
formulated by modifying Condition (2.17) (use $|\lambda| \to 0$ instead of
$|\lambda| \to \infty$) and the proof of Theorem 2.5 in Xiao \cite{X07}.
\end{itemize}
\end{remark}

\vskip 2mm
Now we are ready to state and prove the following theorem, which extends
the uniform Hausdorff and packing dimension results of Hawkes \cite{HJ},
Hawkes and Pruitt \cite{HP}, Perkins and Taylor \cite{PerkinsTaylor}
for stable L\'evy processes to a class of symmetric L\'evy processes.

\begin{theorem}\label{Th:Levy}
Let $X = \{X(t), t \in \R_+\}$ be a symmetric L\'evy process in $\R^d$ with
exponent $\psi(\xi)$. We assume either (i) $1 < \a d$ and (B3) hold; or
(ii) $1 = \a d$ and (B4) hold. Then with probability one,
$$\dim X(E) = \a \, \dim E \ \hbox{ and }\ \dimp X(E) = \a \, \dimp E$$
for all Borel sets $E \subseteq \R_+$.
\end{theorem}

\begin{proof}  This theorem is a consequence of Theorem \ref{main theorem} and
Remark \ref{Re:Lv}. It is sufficient to verify that
\begin{itemize}
\item[(a)] When $1 < \a d$,  (B3) implies Conditions (A1) and (A$2^{\prime\prime}$)
with $H = \frac 1 \alpha$.
\item[(b)] When $1 = \a d$,  (B4) implies Conditions (A1) and (A$3^{\prime\prime}$)
hold with $H = \frac 1 \alpha$.
\end{itemize}
It will be clear that (a) and (b) can be verified by the same method. For simplicity,
we only show (a) in case (i) where $0 < \a < 2$ and leave the rest of the verification
to an interested reader.

In order to verify Conditions (A1), we apply Proposition \ref{Prop:C1-1}.
For any fixed $\gamma\in(0,1/\alpha)$, there exists $\epsilon_0\in (0,2-\alpha)$
such that $\gamma(\alpha+\epsilon_0)<1.$
By (\ref{eq Maximal estimate}), we know that for $t\in(0,1]$ small enough,
\begin{eqnarray*}
\sup_{x\in\mathbb{R}^d}\mathbb{P}^x\bigg\{\sup_{s\in[0,t]}|X(s)-x|\geq t^\gamma\bigg\}
&\leq&    Ct\sup_{|\xi|\leq t^{-\gamma}}|\psi(\xi)|\\
&\leq& Ct \bigg(\sup_{|\xi|\leq \tau}|\psi(\xi)|+K_5t^{-\gamma(\alpha+\epsilon_0)}\bigg)\\
    &\leq&C  t^{1-\gamma(\alpha+\epsilon_0)}.
\end{eqnarray*}
Since $1-\gamma(\alpha+\epsilon_0)>0$, we see that Conditions (A1) holds with
$\eta = 1 - \gamma (\alpha+\epsilon_0)$.
\vskip 0.2cm

In order to verify Condition (A$2^{\prime\prime}$) under \eqref{Eq:growth},
we use an argument from Khoshnevisan and Xiao \cite{KX03, KX08}.
For any $r>0$, consider the nonnegative function
\[
\varphi_r(y) = \prod_{j=1}^d \frac{1 - \cos(2ry_j)} {2 \pi r
y_j^2},\qquad\forall y\in\R^d.
\]
Its Fourier transform is given  by
\begin{equation}\label{eq:FT}
\widehat{\varphi_r}(\xi) = \prod_{j=1}^d \Big( 1 - \frac{
        |\xi_j|}{2r} \Big)^+,\qquad\forall \xi\in\R^d,
\end{equation}
where $a^+=\max(a,0)$.

Note that, for $\xi\in B(0,r)$, we have $1-(2r)^{-1}|\xi_j| \ge \frac 12$.
In light of (\ref{eq:FT}), this implies $\I_{B(0,r)}(\xi) \le 2^d
\widehat{\varphi_r}(\xi)$, where $\I_A$ denotes the indicator function
of the set $A$. On the other hand, if $ |\xi_j| \ge 2r$
for some $j \in \{1, \ldots, d\}$, by (\ref{eq:FT}),  then
$\widehat{\varphi_r}(\xi) =0 $. Hence we have shown that for all
$\forall\xi\in\R^d,$
\begin{equation}\label{Eq:I0}
\I_{B(0,r)}(\xi) \le 2^d \widehat{\varphi_r}(\xi) \le 2^d\,
\I_{B(0, 2 \sqrt{d}\, r)}(\xi).
\end{equation}
Integrating the first inequality in (\ref{Eq:I0}) with respect to
$\nu_t$, the distribution of $X(t)$, and using Parseval's formula
yield
\begin{equation}\label{Eq:I1}
\begin{split}
\PP\big\{|X(t)| \le r\big\} &\le 2^d \int_{\R^d}
\widehat{\varphi_r} (\xi)\,\nu_t(d \xi)\\
& = 2^d \int_{\R^d} \varphi_r(\xi)\, \widehat{\nu_t}(\xi) \,
d\xi\\
&= 2^d \, \int_{\R^d} e^{-t  \psi(\xi)}\,\prod_{j=1}^d \frac{1
- \cos(2r\, \xi_j)} {2 \pi r \xi_j^2}\, d\xi.
\end{split}
\end{equation}
We split the last integral in \eqref{Eq:I1} over $B(0, \tau)
=\{\xi\in \R^d: |\xi| < \tau\}$ and its complement, respectively.
For the first integral, we use the elementary inequality $1 - \cos x \le x^2$
($\forall x \in \R$), to derive
\begin{equation}\label{Eq:I2}
 \int_{|\xi| \le \tau} e^{-t\, \psi(\xi)}\,\prod_{j=1}^d \frac{1 -
\cos(2r\, \xi_j)} {2 \pi r \xi_j^2}\, d\xi \le K \, r^d.
\end{equation}
For the second integral, we use (\ref{Eq:growth}), Parseval's
formula and (\ref{Eq:I0}) to derive
\begin{equation}\label{Eq:I3}
\begin{split}
 \int_{|\xi| > \tau} e^{- t \, \psi(\xi)}\,\prod_{j=1}^d \frac{1 -
\cos(2r\, \xi_j)} {2 \pi r \xi_j^2}\, d\xi  &\le \int_{\R^d}
e^{-K^{-1}_{5}\, t\,  \|\xi\|^{\a - \zeta'}}\,\prod_{j=1}^d
\frac{1 - \cos(2r\, \xi_j)} {2 \pi r \xi_j^2}\, d\xi\\
&= \int_{\R^d} \widehat{\varphi_r} (\xi)\,\mu_t(d \xi)\\
&\le \, \int_{\R^d}  \I_{B(0, 2 \sqrt{d}\, r)}(\xi)\,\mu_t(d \xi)\\
&\le C\, \min\bigg\{1,\, \Big(\frac r {t^{1/(\a
-\zeta')}}\Big)^d\bigg\}.
\end{split}
\end{equation}
In the above, $\mu_t$ denotes the distribution of the isotropic
stable law with characteristic function $\widehat{\mu_t}(\xi) =
e^{-K^{-1}_{5}\, t\,  |\xi|^{\a - \zeta'}}$ and the last inequality
follows from the boundedness and scaling property of the density
function of $\mu_t$. Combining (\ref{Eq:I1})--(\ref{Eq:I3}) we derive
that, for all $0 < t \le T$,  \eqref{Eq:LP1b} holds with $H = 1/\a$ and
$\zeta = \frac{\zeta'}{\alpha (\alpha - \zeta')}$.

Next we verify the lower bound in \eqref{Eq:LP1a} .
Let $\tilde{r} = r/(2\sqrt{d})$.  It follows from \eqref{Eq:I0} that
\begin{equation}\label{Eq:La}
\begin{split}
\PP\big\{|X(t)| \le r\big\} &\ge \int_{\R^d} \widehat{\varphi}_{\tilde{r}}
 (\xi ) \,\nu_t(d \xi) \\
&= \int_{\R^d} \varphi_{\tilde{r}}(\xi)\, \widehat{\nu_t}(\xi) \,
d\xi \\
&=  \int_{\R^d} e^{-t  \psi(\xi)}\,\prod_{j=1}^d \frac{1
- \cos(2{\tilde r}\, \xi_j)} {2 \pi \tilde{r} \xi_j^2}\, d\xi.
\end{split}
\end{equation}
Again, we split the last integral over $B(0, \tau)=\{\xi\in \R^d: |\xi|
< \tau\}$ and its complement, and use a similar argument as in (\ref{Eq:I3})
to obtain
\begin{equation}\label{Eq:Lb}
\begin{split}
\PP\big\{|X(t)| \le r\big\} &\ge \int_{|\xi| \ge \tau }
e^{-K_5 t |\xi|^{\alpha + \zeta'}}\,\prod_{j=1}^d \frac{1
- \cos(2{\tilde r}\, \xi_j)} {2 \pi \tilde{r} \xi_j^2}\, d\xi\\
&= \int_{\R^d } e^{-K_5 t |\xi|^{\alpha + \zeta'}}\,\prod_{j=1}^d \frac{1
- \cos(2{\tilde r}\, \xi_j)} {2 \pi \tilde{r} \xi_j^2}\, d\xi \\
&\qquad - \int_{|\xi|\le \tau } e^{-K_5 t |\xi|^{\alpha + \zeta'}}\,\prod_{j=1}^d \frac{1
- \cos(2{\tilde r}\, \xi_j)} {2 \pi \tilde{r} \xi_j^2}\, d\xi\\
&\ge C\, \min\bigg\{1,\, \Big(\frac r {t^{1/(\a
+\zeta')}}\Big)^d\bigg\}.
\end{split}
\end{equation}
Note that $\frac 1 {\a + \zeta'} = \frac 1 \a - \frac{\zeta'}
{\alpha (\alpha + \zeta')} >  \frac 1 \a - \zeta$.
It follows from (\ref{Eq:Lb}) that we can choose a constant $C_1$
such that for all $0 < t \le T$, the lower bound in \eqref{Eq:LP1a}
holds with $H = 1/\a $ and $\zeta =
\frac{\zeta'}{\alpha (\alpha - \zeta')}$. Hence, we have shown that
\eqref{Eq:growth} implies Condition (A$2^{\prime\prime}$) with
$H = 1/\a $ and $\zeta = \frac{\zeta'}{\alpha (\alpha - \zeta')}$.
This completes the proof of Theorem \ref{Th:Levy}.
\end{proof}

\begin{remark}
Our method for verifying of Condition (A$2^{\prime\prime}$) provides a
comparison theorem for the transition probabilities
of L\'evy processes in terms of their L\'evy exponents. This may be of independent interest.
Several authors have established estimates on the transition density functions
of L\'evy processes based on information on their L\'evy measures or L\'evy exponents;
see Kaleta and Sztonyk \cite{KS15} and the references therein for further information.
\end{remark}

\subsection{Stable jump-diffusions}
Let $X = \{X(t), t \in \R_+, \PP^x, x \in \R^d\}$ be a Feller
process with values in $\R^d$ corresponding to the Feller
semigroup defined by the following equation:
\begin{equation}\label{Eq:Equ}
\frac{\partial u}{\partial t} = \langle A(x), \frac{\partial
u}{\partial x} \rangle + \int_0^\infty \int_{\mathbb{S}_d} \Big(
u(x + |\xi|s) -u(x) -\frac{\langle |\xi|s, \frac{\partial u}{\partial
x}(x) \rangle} {1 + |\xi |^2}\Big)\, \frac{d |\xi|}{|\xi|^{1
+\a}}\, \widetilde{\mathsf{M}}(x, ds),
\end{equation}
where the drift $A$ and the spectral measure
$\widetilde{\mathsf{M}}$ on $\mathbb{S}_d$ depend smoothly on $x$,
see Theorem 3.1 of Kolokoltsov \cite{Kolok00b} for the precise conditions
on $A$ and $\widetilde{\mathsf{M}}$.

Following Kolokoltsov \cite{Kolok00b}, we call $X$ a stable jump-diffusion.
Roughly speaking, these are the processes corresponding to stable
L\'evy processes in the same way as the ordinary diffusions
corresponding to Brownian motion.

Locally, the stable jump-diffusion $X$ resembles a stable L\'evy
process, hence it is expected that a stable jump-diffusion has
sample path properties similar to those of a stable L\'evy
process. Some of these properties such as the limsup behavior of
$X(t)$ as $t \to 0$ have been established by Kolokoltsov \cite[section 6]{Kolok00b}.
Moreover, for every fixed Borel set $E \subseteq \R_+$, the Hausdorff
dimension of the image set $X(E)$ can be
derived from Theorem 4.14 in Xiao \cite{X2004}.

The following theorem proves a uniform Hausdorff and packing dimension result for
stable jump diffusions.

\begin{theorem}\label{Th:Kolo}
Let $X=\{X(t), t \in \R_+, \PP^x\}$ be a stable jump-diffusion in
$\R^d$ with index $\a \in (0, 2]$ as described above. If $\alpha \le d$, then for every
$x \in \R^d$, $\PP^x$-almost surely
\begin{equation}\label{Eq:Kolo}
\dim X(E) = \a \,\dim E \ \hbox{ and }\ \dimp X(E) = \a \,\dimp E
\end{equation}
hold for all Borel sets $E \subseteq \R_+$.
\end{theorem}

\begin{proof}\ It follows from Theorem 6.1 of Kolokoltsov \cite{Kolok00b}
that Condition (A1) holds for $X$ with $H = 1/\a$. On the other
hand, Theorem 3.1 of \cite{Kolok00b} implies that (A2) with $H =
1/\a$, $\ep = 0$ and $\zeta = 0$. Hence (\ref{Eq:Kolo}) follows immediately
from Theorem \ref{main theorem}.
\end{proof}

\begin{remark}\label{Remark section 4.2}
We remark that one can also apply the inequality (\ref{eq Maximal estimate}) to verify Condition (A1) for this case.
By (1.9) of \cite{Kolok00b}, the symbol of the stable jump diffusion defined by (\ref{Eq:Equ}) has the form:
$$
q(x,\xi)=i(A(x),\xi)-\int_{\mathbb{S}_d}|(\xi,s)|^\alpha M(x,ds).
$$
Moreover,  we assume that same  conditions on $A$ and $M$ as in Theorem 3.1 in \cite{Kolok00b}, which contain
\begin{itemize}
  \item there exist positive constants $C_1$ and $C_2$ such that
        $$
        C_1|\xi|^\alpha\leq \int_{\mathbb{S}_d}|(\xi,s)|^\alpha M(x,ds)\leq C_2|\xi|^\alpha,
        $$
  \item $A$ is uniformly bounded in $x$, i.e. $\sup_{x\in\mathbb{R}^d}|A(x)|<\infty$,

  \item $A(x)\equiv 0$ for $\alpha\leq 1$.
\end{itemize}
Then, by  (\ref{eq Maximal estimate}), for any $\gamma\in(0,1/\alpha)$, we have, for all $x \in \R^d$ and  $t\in[0,1]$,
\[
\begin{split}
\mathbb{P}^x\bigg\{\sup_{s\in[0,t]}|X_s-x|\geq t^\gamma\bigg\}
&\leq
Ct  \sup_{|y-x|\leq t^\gamma}\sup_{|\xi|\leq t^{-\gamma}}|q(y,\xi)|\\
&\leq
Ct \sup_{|\xi|\leq t^{-\gamma}}\Big(\sup_{x\in\mathbb{R}^d}|A(x)||\xi|+C_2|\xi|^\alpha\Big)\\
&\leq
Ct^{1-\gamma \alpha}\Big(\sup_{x\in\mathbb{R}^d}|A(x)|+C_2\Big).
\end{split}
\]
This verifies Condition (A1) with $\eta = 1 - \gamma \a$.
\end{remark}

\subsection{Non-symmetric stable-type processes}

Let $X=\{X(t), t \in \R_+\}$ be a pure jump process such that its infinitesimal
generator has the following form:
$$
\mathcal{L}^{\kappa}_{\alpha}f(x):=\lim_{\varepsilon\rightarrow 0}\int_{\{z\in \mathbb{R}^d:|z|
\geq \varepsilon\}}(f(x+z)-f(x))\frac{\kappa(x,z)}{|z|^{d+\alpha}}dz,
$$
where $d\geq 1, 0<\alpha<2$, and $\kappa(x, z)$ is a measurable function on
$\mathbb{R}^d\times\mathbb{R}^d$ satisfying
$$
0 <\kappa_0\leq \kappa(x, z)\leq \kappa_1,\quad \kappa(x, z)=\kappa(x,-z),
$$
and for some $\beta\in(0, 1)$
$$
|\kappa(x, z)-\kappa(y, z)| \leq \kappa_2|x-y|^{\beta}.
$$
This class of Markov processes has been studied by \cite{CK, CZ}, among others.

The following uniform dimension result  is a consequence of  Theorem \ref{main theorem}.

\begin{thm} \label{stable-process} Let $X=\{X(t), t \in \R_+\}$ be an non-symmetric $\alpha$-stable-type
Markov process defined above. If $1 < \a d$,  then,
for every $x \in \R^d$, $\PP^x$-almost surely
\begin{eqnarray}
\dim X(E) = \a \,\dim E \ \hbox{ and }\ \dimp X(E) = \a \,\dimp E
\end{eqnarray}
hold for all Borel sets $E \subseteq \R_+$.
\end{thm}

\begin{proof}
By \cite{CZ, Jin}, $X_t$ has a H\"{o}lder continuous transition density
 function $p(t,x,y)$. Furthermore, for any $T>0$, there exists a constant $C>1$ depending on $d, \alpha, \beta, \kappa_0, \kappa_1, \kappa_2, T$ such that for all $t\in (0,T], x, y\in \mathbb{R}^d$,
\begin{eqnarray}\label{density for stable like}
C^{-1} \min\Big\{t^{-d/\alpha}, \frac{t}{|x-y|^{d+\alpha}}\Big\}\leq p(t,x,y)
\leq C\min\Big\{t^{-d/\alpha}, \frac{t}{|x-y|^{d+\alpha}}\Big\}.
\end{eqnarray}

Define
$$
P_1(t,x, A):=C^{-1}\int_{A}\min\Big\{t^{-d/\alpha}, \frac{t}{|x-y|^{d+\alpha}}\Big\}dy
$$
and
$$
P_2(t,x, A):=C\int_{A}\min\Big\{t^{-d/\alpha}, \frac{t}{|x-y|^{d+\alpha}}\Big\}dy.
$$
Both $P_1$ and $P_2$ have $1/\alpha$-self-similar property, i.e., for any $r>0$
$$
P_i(t,x, A)=P_i(rt,r^{1/\alpha}x,r^{1/\alpha}A), i=1,2.
$$
Indeed, for all $A \in \mcl B(\R^d)$ and $r>0$, we have
\Bes
\begin{split}
& \ \ \ \int_{r^{1/\alpha} A} \min\Big\{(rt)^{-d/\alpha},
\frac{rt}{|r^{1/\alpha}x-y|^{d+\alpha}}\Big\} \dif y \\
&=\int_A  \min\Big\{(rt)^{-d/\alpha}, \frac{rt}
{r^{(d+\alpha)/\alpha}|x-z|^{d+\alpha}}\Big\} r^{d/\alpha}\dif z \\
&=\int_A  \min\Big\{t^{-d/\alpha}, \frac{t}{|x-z|^{d+\alpha}}\Big\} \dif z.
\end{split}
\Ees
By a straightforward computation, there exists constant $C>1$ such that
$$
P_1(1,y, B(x,r))\geq C^{-1} \min\{1, r^{d}\}, \quad \forall |y-x|\leq r,\quad  r>0.
$$
and
$$
P_2(1,x, B(x,r))\leq C \min\{1, r^{d}\},\quad \forall x\in\mathbb{R}^d,\quad r>0.
$$
Then for all $t\in (0, T]$ and $|y-x|\leq r$,
\begin{equation*}
\begin{split}
P(t, y, B(x,r)) &\geq P_1(t,y,B(x,r))\\
&= P_1(1, t^{-\frac{1}{\alpha}}y, B(t^{-\frac{1}{\alpha}}x, t^{-\frac{1}{\alpha}}r))\\
&\ge C^{-1}\min\bigg\{1,\, \Big(\frac{r} {t^{\frac{1}{\alpha}}}\Big)^{d}\bigg\}
\end{split}
\end{equation*}
and
\[
\begin{split}
P(t, x, B(x,r))&\leq P_2(t,x,B(x,r))\\
&= P_2(1, t^{-\frac{1}{\alpha}}x, B(t^{-\frac{1}{\alpha}}x, t^{-\frac{1}{\alpha}}r))\\
&\le C \min\bigg\{1,\, \Big(\frac{r}
{t^{\frac{1}{\alpha}}}\Big)^{d}\bigg\}.
\end{split}
\]
Thus Condition  (A2)  holds.

\vskip 0.2cm
Next, we verify the condition (A1). As in Remark \ref{Remark section 4.2} , this can be done by applying
Proposition \ref{Prop:C1-1}  and the fact that the symbol $q$ of $X$  has the form
$$
q(x,\xi)=\int_{\mathbb{R}^d\backslash \{0\}} \Big(1-e^{iz\cdot\xi}+iz\cdot\xi 1_{(0,1]}(|z|)\Big)\frac{\kappa(x,z)}{|z|^{d+\alpha}}dz,
$$
which can be bounded from above by $C |\xi|^\alpha$. Here we provide a different proof.

By the L\'evy-It\^o decomposition, one has
$$
X_t=X_0+\int^t_0 \int_{|y|\leq l}y \tilde N(X_{s-}, dy, ds)+\int^t_0 \int_{|y|>l}y N(X_{s-}, dy, ds),
$$
where $l$ can be any positive constant, $N(x, dy, ds)$ is the
Poisson random measure with the intensity measure $\nu(x, dy) ds:=\frac{\kappa(x,y)}{|y|^{d+\alpha}}dyds$ and
$\tilde N(X_{s-}, dy, ds)=N(X_{s-}, dy, ds)-\nu(X_{s-}, dy) ds$ is
the compensated Poisson random measure.

Since $\alpha\in(0,2)$,
we can find some $p$ satisfying $0<p\leq1$ and $p<\alpha<2p$.
For some $l>0$ and $x \in \R^d$, we define a smooth function
$f$ on $\mathbb{R}^d$ by
$$f(y)=(|y-x|^{2}+l^2)^{p/2}, \quad  y \in\mathbb{R}^d,$$
it is easy to check that for all $x,y\in\mathbb{R}^d$,
\begin{eqnarray}
|f(y_1)-f(y_2)| \leq|y_1-y_2|^{p}.     \label{f}
\end{eqnarray}
Now let $X_0=x$, by It\^o's formula \cite[Sect. 4.4.2]{Ap04},
we get, for any $l>0$,
\begin{eqnarray}
f(X_t)=\!\!\!\!\!\!\!\!&&f(x)+\int^{t}_{0}\int_{|y|\leq l}
[f(X_{s-}+y)-f(X_{s-})]\widetilde{N}(X_{s-}, dy,ds)\nonumber\\
&&+\int^{t}_{0}\int_{|y|\leq l} \left[f(X_{s-}+y)-f(X_{s-})
-\langle \nabla f(X_{s-}), y\rangle\right]\nu(X_{s-}, dy)ds\nonumber\\
&&+\int^{t}_{0}\int_{|y|>l}[f(X_{s-}+y)-f(X_{s-})]N(X_{s-},dy,ds)\nonumber\\
=:\!\!\!\!\!\!\!\!&&l^p+I_1(t)+I_2(t)+I_3(t) \nonumber.
\end{eqnarray}

Take $l=T^{1/\alpha}$, for any $T>0$, by Burkholder's inequality and (\ref{f}),
we have
\begin{eqnarray}
\mathbb{E}\bigg(\sup_{0\leq t\leq T}|I_1(t)|\bigg)
&\leq& \mathbb{E}\left[\int^{T}_{0}\int_{|y|\leq T^{1/\alpha}}|y|^{2p}N(X_{s-},dy, ds)\right]^{1/2}\nonumber\\
&\leq& \left[\int_0^T \int_{|y|\leq T^{1/\alpha}}|y|^{2p}
\frac{\kappa_1}{|y|^{d+\alpha}}dy ds \right]^{1/2} \nonumber\\
&\leq& CT^{p/\alpha}.\label{I_1}
\end{eqnarray}
For $I_2(t)$, by Taylor's expansion, we have
\
\Be
\begin{split}
f(X_{s-}+y)-f(X_{s-})
-\langle \nabla f(X_{s-}), y\rangle=\langle y, \nabla^2 f(X_{s-}+\theta y)y \rangle,
\end{split}
\Ee
where $\theta \in [0,1]$ depending on $X_{s-}$ and $y$, one can verify that
\Bes
\begin{split}
&\left|\langle y, \nabla^2 f(X_{s-}+\theta y)y \rangle\right|\\
&=\bigg|\frac{p|y|^2} {\left(T^{2/\alpha}+|X_{s-}+\theta y-x|^2\right)^{1-\frac p2}}
+\frac{p(p-2)|\langle y, X_{s-}+\theta y-x\rangle|^2}{\left(T^{2/\alpha}+|X_{s-}
+\theta y-x|^2\right)^{2-\frac p2}}\bigg| \\
& \le \frac{p|y|^2}{T^{\frac 2\alpha (1-\frac p2)}}+\frac{p(2-p)|y|^2 |X_{s-}
+\theta y-x|^2}{\left(T^{2/\alpha}+|X_{s-}+\theta y-x|^2\right)^{2-\frac p2}} \\
& \le \frac{(3p-p^2)|y|^2}{T^{\frac 2\alpha (1-\frac p2)}}.
\end{split}
\Ees
Hence,
\begin{equation}\label{I_2}
\begin{split}
\mathbb{E}\bigg(\sup_{0\leq t\leq T}|I_2(t)|\bigg)
&\leq C T^{\frac{2}{\alpha}(\frac{p}{2}-1)}\int_0^T
\E \int_{|y|\leq T^{1/\alpha}}|y|^{2}\nu(X_{s-},dy) ds \nonumber\\
& \le  C T^{\frac{2}{\alpha}(\frac{p}{2}-1)}\int_0^T
\int_{|y|\leq T^{1/\alpha}}\frac{\kappa_1 |y|^{2}}{|y|^{d+\alpha}} dy ds  \\
&\leq  CT^{p/\alpha}.
\end{split}
\end{equation}
For $I_3(t)$, by (\ref{f}) again, we get
\begin{equation}\label{I_3}
\begin{split}
\mathbb{E}\bigg(\sup_{0\leq t\leq T}|I_3(t)|\bigg) &\leq \mathbb{E} \int^{T}_{0}
\int_{|y|>T^{1/\alpha}}|y|^{p}\nu(X_{s-},dy)ds  \\
&\leq \int_0^T \int_{|y|>T^{1/\alpha}}\frac{\kappa_1 |y|^{p}}{|y|^{d+\alpha}} dy ds\\
&\leq CT^{p/\alpha}.
\end{split}
\end{equation}
Combining \eqref{I_1}-\eqref{I_3} yields
\begin{eqnarray*}
\mathbb{E}^x\bigg(\sup_{t\in[0,T]}\big|(|X_t-x|^{2}+
T^{\frac{2}{\alpha}})^{p/2}-T^{\frac{p}{\alpha}}\big| \bigg) \leq CT^{p/\alpha}.
\end{eqnarray*}
Hence, for any $r>\alpha$,
\begin{eqnarray*}
\mathbb{P}^{x} \bigg\{\sup_{0\leq s\leq t}|X_s-x|\geq t^{1/r} \bigg\}
&\leq& \frac{\mathbb{E}^x\left[\sup_{s\in[0,t]}|X_s-x|^{p}\right]}{ t^{p/r}}\\
&\leq&\frac{\mathbb{E}^x\left[\sup_{s\in[0,t]}|(|X_s-x|^{2}+t^{\frac{2}{\alpha}})^{p/2}
-t^{\frac{p}{\alpha}}\right]+t^{\frac{p}{\alpha}}}{ t^{p/r}}\\
&\leq& C t^{p/\alpha-p/r}.
\end{eqnarray*}
Thus Condition (A1) also holds. Hence, the conclusion of Theorem \ref{stable-process}
follows from Theorem \ref{main theorem}.
\end{proof}

\begin{remark}
Notice that the estimate of transition density $p(t,x,y)$ in (\ref{density for stable like}) holds for any $t\in (0, T]$, which is sufficient to prove condition (A2). But for condition (A3), we need the estimate (\ref{density for stable like}) holds for any $t>0$. Hence the above theorem holds only for $1<\alpha d$.

\end{remark}


\medskip

\textbf{Acknowledgment}. We thank the referee for his or her valuable comments and,  in particular,
for suggesting the use of Proposition 3.2, which has helped to simplify the proof of Theorem
\ref{Th:Levy} and improve the quality of this paper.

This work was conducted during the first author visited
the Department of Mathematics, Faculty of Science and Technology, University
of Macau, he thanks for the finance support and hospitality. Xiaobin Sun is
supported by the National Natural Science Foundation of China (11601196, 11771187),
Natural Science Foundation of the Higher Education Institutions of Jiangsu Province (16KJB110006), Scientific Research Staring Foundation of Jiangsu Normal University (15XLR010)
and the Project Funded by the Priority Academic Program Development of Jiangsu
Higher Education Institutions. Y. Xiao is partially supported by NSF
grants DMS-1612885 and DMS-1607089. Lihu Xu is supported by the following grants:
NNSFC 11571390, Macau S.A.R. FDCT 030/2016/A1 and FDCT 038/2017/A1,
University of Macau MYRG (2015-00021-FST, 2016-00025-FST). Jianliang Zhai is
supported by the National Natural Science
Foundation of China (11431014, 11401557).

\bibliographystyle{amsplain}

\end{document}